\documentclass[12pt,a4paper]{amsart}
\usepackage[english]{babel}
\usepackage{amsmath}
\usepackage{amsfonts}
\usepackage{amssymb}
\usepackage{amsthm}
\usepackage{graphicx}
\usepackage{enumitem}
\usepackage{stmaryrd}
\usepackage{hyperref}
\usepackage{bbm}
\usepackage{xcolor,mathtools}
\usepackage{mathrsfs}
\usepackage{esint}

\theoremstyle{definition}
\newtheorem{defi}{Definition}[section]
\newtheorem{definition}[defi]{Definition}
\newtheorem{ex}[defi]{Example}

\newtheorem{rem}[defi]{Remark}
\theoremstyle{plain}
\newtheorem{teo}[defi]{Theorem}
\newtheorem{theorem}[defi]{Theorem}

\newtheorem{lemma}[defi]{Lemma}
\newtheorem{pro}[defi]{Proposition}

\newtheorem*{teorema}{Theorem}
\usepackage[left=2cm,right=2cm,top=2cm,bottom=2cm]{geometry}
\newcommand{\N}{\mathbb{N}}
\newcommand{\he}{\mathbb{H}}
\newcommand{\hel}{\mathfrak{h}}

\newcommand{\R}{\mathbb{R}}
\newcommand{\s}{\mathcal{S}}

\newcommand{\pp}{\partial}

\newcommand{\V}{\mathbb{V}}
\newcommand{\W}{\mathbb{W}}

\newcommand{\ve}{\varepsilon}

\newcommand{\se}{\subseteq}

\newcommand{\ceq}{\coloneqq}

\renewcommand{\H}{\mathbb{H}}
\newcommand{\res} {\mathop{\hbox{\vrule height 7pt width .5pt depth 0pt \vrule height .5pt width 6pt depth 0pt}}\nolimits}

\newcommand{\G}{\mathbb{G}}

\newcommand{\leb}{\mathcal{L}}
\newcommand{\gr}{\operatorname{gr}}
\newcommand{\Lip}{\operatorname{Lip}}
\newcommand{\lip}{\operatorname{lip}}
\newcommand{\T}{\mathbb{T}}
\newcommand{\diam}{\operatorname{diam}}
\newcommand{\h}{\mathbb{H}}

\title{Stepanov Differentiability Theorem for intrinsic graphs in Heisenberg groups}

\author{Marco Di Marco}
\address{Dipartimento di Matematica ``T. Levi-Civita'', Università di Padova, via Trieste 63, 35121 Padova, Italy.}
\email{marco.dimarco@phd.unipd.it}

\author{Andrea Pinamonti}
\address{Dipartimento di Matematica, Università di Trento, via Sommarive, 14, 38123 Povo (TN), Italy}
\email{andrea.pinamonti@unitn.it}

\author{Davide Vittone}
\address{Dipartimento di Matematica ``T. Levi-Civita'', Università di Padova, via Trieste 63, 35121 Padova, Italy.}
\email{davide.vittone@unipd.it}

\author{Kilian Zambanini}
\address{Dipartimento di Matematica, Università di Trento, via Sommarive, 14, 38123 Povo (TN), Italy}
\email{kilian.zambanini@unitn.it}

\begin{document}

\begin{abstract}
We prove a Stepanov differentiability type theorem for intrinsic  graphs in sub-Riemannian Heisenberg groups.
\end{abstract}

\maketitle
\section{Introduction}

The classical Rademacher theorem states that: if $\Omega\subset\mathbb{R}^n$ is an open set and $f:\Omega\to \mathbb{R}^m$ is Lipschitz continuous, then $f$ is differentiable almost everywhere in $\Omega$. One important generalization of  Rademacher's theorem is the following result due to Stepanov (see e.g. \cite[3.1.8]{federer1969} and \cite{stepanov1923}) 
\begin{teorema}\label{teo_classStep}
Let $\Omega\se \R^n$ be an open set and let $f:\Omega \to \R^m$. Consider the set
\begin{equation}\label{def_sf}
S_f \ceq \left\lbrace a \in \Omega: \limsup_{x \to a}\frac{|f(x)-f(a)|}{|x-a|}<+\infty \right\rbrace.
\end{equation}
Then $f$ is differentiable almost everywhere on $S_f$.
\end{teorema}
Roughly speaking, the Stepanov differentiability theorem generalizes the classical notion of differentiability by relaxing the conditions under which a function is differentiable almost everywhere, broadening the scope of functions to include those that may not exhibit pointwise regularity. 

The classical proof of the Stepanov differentiability theorem can be found in many  textbooks, see e.g. \cite{federer1969} and it is essentially based on density theorems and the application of  Rademacher's theorem to a Lipschitz extension of $f_{|G_n}$, where $G_n$ are suitable measurable sets on which $f$ is Lipschitz. In 1999, J. Mal\'y \cite{maly1999} proposed, for \emph{real-valued} functions defined on separable Banach spaces, an alternative and elegant proof without using any density theorem. 
 In \cite{MR1651959,MR2414208}, using differentiability points of the distance functions instead of density theorems, the authors were able to prove a Stepanov  type theorem (in the Gâteux sense) for functions between infinite dimensional Banach spaces.
Later, in \cite{maly2015}, Mal\'y  and Zaj{\'{\i}}{\v{c}}ek, presented a new approach which shows how a Stepanov  type theorem (in the Frechet sense) can be inferred from the corresponding theorem of Rademacher type.  

\vspace{0.5cm}

In recent years, there has been significant and ongoing research aimed at extending classical analysis results from Euclidean spaces to more general metric-measure spaces (see e.g. \cite{MR3363168,MR1800917} and references therein). A major result is due to Cheeger \cite{MR1708448}, who found a deep generalization of  Rademacher's theorem in the context of doubling metric measure spaces that satisfy a Poincarè inequality. More recently, the approach in \cite{maly1999} has been used in \cite{brz2004} to prove a Stepanov  type theorem for real-valued maps defined on metric spaces endowed with a doubling Borel measure and later in \cite{wz2015} the result has been further generalized to maps between metric measure spaces and Banach spaces.
Building on the foundation established in \cite{MR1708448}, a substantial body of literature has emerged, see e.g. \cite{MR3363168, MR2291675} for a complete overview. Let us also mention the recent~\cite{dedonato2024stepanovtheoremqvaluedfunctions}.

The notion of Lipschitz submanifolds in sub-Riemannian geometry was introduced, at least in the setting of Carnot groups, by B. Franchi, R. Serapioni and F. Serra Cassano in \cite{MR1871966,fssc2006, fs2016} through the theory of intrinsic Lipschitz graphs. Roughly speaking, a subset $S\subset \mathbb{G}$ of a Carnot group $\mathbb{G}$ is intrinsic Lipschitz if at each point $P\in S$ there is an intrinsic cone with vertex $P$ and fixed opening intersecting $S$ only at $P$. Remarkably, this notion turned out to be the right one in the setting of the intrinsic rectifiability in the simplest Carnot group, namely the Heisenberg group $\mathbb{H}^n$. Indeed, it was proved in \cite{MR1871966,vittone2021} that the notion of rectifiable set in terms of an intrinsic regular hypersurfaces is equivalent to the one in terms of intrinsic Lipschitz graphs. Recently, the theory of intrinsic Lipschitz functions has played a crucial in the study of quantitative rectifiability \cite{MR4388340,MR3992573} and has even been applied to problems in information theory \cite{MR3815462,MR4460594}. See also \cite{MR3465805, sc2016} for further applications.

The main open question in this area of research is whether a Rademacher-type theorem holds for intrinsically Lipschitz functions between homogeneous subgroups of a Carnot group. Specifically, consider a splitting $\mathbb{G}=\mathbb{W}\mathbb{V}$
of a Carnot group $\mathbb{G}$
and let $\phi:\mathbb{W}\to \mathbb{V}$
 be an intrinsically Lipschitz function. The question is whether such a function is intrinsically differentiable almost everywhere (see Definition \ref{intdiff} below).
In \cite{fms2014,fs2016}, the authors provided a positive answer when $\mathbb{G}$ has step two or is a Carnot group of type $\star$ and $\mathbb{V}\equiv\mathbb{R}$. More recently, the third named author \cite{vittone2021} has proved that the answer is also affirmative in the case of the Heisenberg group, even without assuming any prior splitting condition. We also address the interested reader to \cite{MR4277829,CMP24, MR4329286} for further results. Remarkably, in \cite{MR4379561}, the authors constructed intrinsic Lipschitz graphs of codimension $2$ in Carnot groups which are not intrinsically differentiable almost everywhere thus discovering a deep connection between the notion of intrinsic differentiability and the geometry of the underlying Carnot group. 

In the present paper we prove an analogue of Stepanov differentiability theorem for intrinsic graphs in Heisenberg groups. More precisely, we prove (see Theorems \ref{stepanov1} and \ref{maindiff} below) that, given  complementary subgroups $\W,\V$ of $\H^n$, then every function $\phi:A \se \W \to \V$ is intrinsically differentiable almost everywhere on the set of points where its pointwise intrinsic Lipschitz constant is finite.
The proof of both results, although inspired by that of \cite{federer1969}, differs in some fundamental points. First of all, the notion of intrinsic Lipschitz continuity differs from the classical metric one (see Definition \ref{def_lipschitz}). This means that an intrinsic Lipschitz function $\phi:A\subseteq \W\to\V$ need not to be metric Lipschitz (with respect to the distances induced on $\W$ and $\V$), see \cite{fssc2010}. Therefore our result does not fit into the classical framework of Lipschitz maps between metric measure spaces. 
Secondly, the notions of differentiability and Lipschitz continuity are indeed intrinsic geometric properties of $\phi$, which take into account not only the structure of $\W,\V$ but also how they interact inside $\he^n$. 
To overcome these problems, we revisited the proof provided in \cite{federer1969} from a new geometric perspective; specifically, instead of working with the function $\phi$, we considered its intrinsic graph. Finally, in the last part of the paper, we provide an alternative proof of the theorem in the  case of codimension one, using the approach developed in \cite{maly1999}. We point out that the proof of our main results, i.e., Theorems \ref{stepanov1} and \ref{maindiff}, are not dependent on the particular structure of $\he^n$. They can be extended to intrinsic graphs in general Carnot groups $\G$ endowed with a splitting $\W\V$ for which a Rademacher theorem holds.

\vspace{1cm}

\textbf{Acknowledgments:} 
The authors would like to thank Francesco Serra Cassano and Raul Serapioni for many interesting discussions on the topic of the present paper.\\
The authors are members and acknowledge the support of the Istituto Nazionale di Alta Matematica (INdAM), Gruppo Nazionale per l'Analisi Matematica, la Probabilità e le loro Applicazioni (GNAMPA).\\
 M. Di Marco and  D. Vittone are supported by University of Padova. D. Vittone is also supported by PRIN 2022PJ9EFL project ``Geometric Measure Theory: Structure of Singular Measures, Regularity Theory and Applications in the Calculus of Variations'' funded by the European Union - Next Generation EU, Mission 4, component 2 - CUP:E53D23005860006.\\
 A. Pinamonti and K. Zambanini are supported by MIUR-PRIN 2022 Project \emph{Regularity problems in sub-Riemannian structures}  Project code: 2022F4F2LH.\\

\section{Notation and preliminary results}
\begin{defi}
For $n\geq 1$ we denote by $\he^n$ the $n$-th {\em Heisenberg group}, identified with $\R^{2n+1}$ through exponential coordinates. We denote a point $p\in \he^n$ as $p=(x,y,t)$ with $x,y \in \R^{n}$ and $t \in \R$. If $p=(x,y,t),\, q=(x',y',t') \in \he^n$, the group operation is defined as 
\[
p \cdot q \ceq (x+x',y+y',t+t'+\tfrac{1}{2}\langle x,y' \rangle_{\R^n}-\tfrac{1}{2}\langle x', y \rangle_{\R^n}).
\]
If $p=(x,y,t)\in\he^n$, its inverse is $p^{-1}=(-x,-y,-t)$ and $0=(0,0,0)\in \he^n$ is the identity of the group. 

For $\lambda >0$, we denote by $\delta_\lambda:\he^n\to\he^n$ the \textit{intrinsic dilations} of the Heisenberg group defined, for $p=(x,y,t)\in\he^n$, by $\delta_\lambda(x,y,t) \ceq (\lambda x,\lambda y, \lambda^2 t)$. Observe that dilations form a one-parameter family of group isomorphisms.
We say that a subgroup of $\he^n$ is \textit{homogeneous} if it is closed under intrinsic dilations.\\
The Heisenberg group $\he^n$ admits the structure of a Lie group of \textit{topological dimension} $2n+1$.
We denote by $Q \ceq 2n+2$ the {\em homogeneous dimension} of $\he^n$. The Lebesgue measure $\mathcal L^{2n+1}$ is the Haar measure on $\he^n\equiv\R^{2n+1}$ and it is $Q$-homogeneous with respect to dilations.
\end{defi}

\begin{defi}
We denote by $\hel^n$ (or by $\hel$ when the dimension $n$ is clear) the $(2n+1)$-dimensional Lie algebra of  left invariant vector fields in $\he^n$. The algebra $\hel$ is generated by the vector fields $X_1,...,X_n,Y_1,...,Y_n,T$, where (for $1 \leq j \leq n$)
\[
X_j \ceq \pp_{x_j}-\frac{y_j}{2}\pp_t,
\qquad
Y_j \ceq \pp_{y_j}+\frac{x_j}{2}\pp_t,
\qquad
T \ceq \pp_t.
\]
We denote by $\hel_1$ the horizontal subspace of $\hel$, i.e.,
\[
\hel_1 \ceq \operatorname{span}(X_1,...,X_n,Y_1,...,Y_n),
\]
and by $\hel_2$ the linear span of $T$. Since $[X_j,Y_j]=T$, the Lie algebra $\hel$ admits the 2-step stratification $\hel=\hel_1 \oplus \hel_2$. 
Note also that, since $\he^n$ is simply connected and nilpotent, the exponential map exp:$\,\hel\to\H^n$ is a global diffeomorphism.

\end{defi}
\begin{defi} 
We say that a distance function $d:\H^n \times \H^n \to [0,+\infty)$ is a \emph{left invariant and homogeneous distance} if
\begin{enumerate}[label=(\roman*)]
\item $d(p,q)=d(r \cdot p,r \cdot q)$ for all $p,q,r \in \H^n$,
\item $d(\delta_\lambda(p),\delta_\lambda(q))=\lambda d(p,q)$ for all $p,q \in \H^n$ and $\lambda>0$.
\end{enumerate}
We define the  norm $\|\cdot \|$ associated to $d$ as $\|p\|\ceq d(0,p)$ for every $p \in \H^n$. Moreover, if for every $(x,y,t),(x',y',t) \in \H^n$ such that $|(x,y)|_{\R^{2n}}=|(x',y')|_{\R^{2n}}$ we have $\|(x,y,t)\|=\|(x',y',t)\|$ we say that $d$ is a \emph{rotationally invariant} distance.

\end{defi}

\begin{ex}\label{ex_dist}
There are numerous examples of left invariant, homogeneous and rotationally invariant distances on $\H^n$, the most noteworthy being the following.
\begin{enumerate}[label=(\roman*)]
\item The \emph{Carnot-Carathéodory} distance $d_{cc}$ defined for $p \in \H^n$ as
\[
d_{cc}(0,p)\ceq \inf\left\{\|h\|_{L^1([0,1],\R^{2n})} : 
\begin{array}{l}
\text{the curve $\gamma_h:[0,1]\to\H^n$ defined by}\\
\gamma_h(0)=0,\ \dot\gamma_h=\sum_{j=1}^n(h_jX_j+h_{j+n}Y_j)(\gamma_h)\\
\text{has final point }\gamma_h(1)=p
\end{array}
\right\}.
\]
\item The \emph{infinity} distance $d_\infty$ defined for $(x,y,t) \in \H^n$ as
\[
d_\infty(0,(x,y,t)) \ceq \max \lbrace |(x,y)|_{\R^{2n}},2|t|_\R^{\frac{1}{2}}\rbrace.
\]
\item The \emph{Korányi (or Cygan-Korányi)} distance $d_K$ defined for $(x,y,t) \in \H^n$ as
\[
d_K(0,(x,y,t)) \ceq \left((|x|_{\R^n}^2+|y|_{\R^n}^2)^2+16t^2 \right)^\frac{ 1}{4}.
\]
\end{enumerate}
\end{ex}

In view of the following proposition, we will denote by $d$ a fixed left invariant, homogeneous and rotationally invariant distance on $\H^n$, by $\|\cdot\|$ its associated norm and for any $p \in \H^n,r>0$ we will denote by $B(p,r)$ the corresponding open balls.

\begin{pro}[{\cite[Proposition 1.3.15]{didonatophd}}]\label{pro_distequiv}
Let $d_1$ and $d_2$ be left invariant and homogeneous distances on $\H^n$. Then they are bi-Lipschitz equivalent, i.e., there exists $C>0$ such that for all $p,q \in \H^n$ 
\[
\frac{1}{C}d_2(p,q)\leq d_1(p,q) \leq C d_2(p,q).
\]
In particular every left invariant and homogeneous distance induces the Euclidean topology on $\H^n$.
\end{pro}
We say that a homogeneous subgroup $\V\subset\he^n$ is \textit{horizontal} if it is contained in the \textit{horizontal fiber} i.e. $\V\subseteq \text{exp}(\hel_1)$. Horizontal subgroups can be identified with $(\R^k,|\cdot|)$: more precisely, if $V_1,\dots,V_k\in \hel_1$ are such that $\V=\exp(\text{span}(V_1,\dots,V_k))$, then the map $\R^k\ni x\mapsto\exp(x_1V_1+\dots+x_kV_k)$ is a biLipschitz diffeomorphism between $(\R^k,|\cdot|)$ and $(\V,d)$. In particular the Hausdorff dimension of $\V$ equals the topological dimension $k$.\\
On the other hand, we will say that a subgroup $\W$ is \textit{vertical} if it contains the \textit{center} of the group i.e.
exp$(\hel_2)\subseteq \W$. In this case the Hausdorff dimension of $\W$ is greater than the topological one: for instance the metric dimension of $\he^n$ coincides with the homogeneous dimension $Q=2n+2$.\\ All homogeneous subgroups of the Heisenberg group are either horizontal (and in this case they are \textit{abelian}) or vertical (and in this case they are \textit{normal}).
\begin{defi}
Let $\W,\V$ be homogeneous subgroups of $\H^n$. We say that $\W,\V$ are \textit{complementary subgroups} in $\H^n$ if $\W\cap \V=\{ 0 \}$ and $\W\cdot\V=\H^n$. 
\end{defi}
All possible couples of complementary subgroups in $\he ^n$ are formed by an (abelian) horizontal subgroup $\V$ of dimension $k$, for $1\leq k\leq n$, and by a (normal) vertical subgroup $\W$ of dimension $2n+1-k$. 
\begin{rem}\label{rem_componenti}
If $\W$ and $\V$ are complementary subgroups in $\H^n$, then each element $p\in\H^n$ can be written in a unique way as $p=w\cdot v$, for $w\in\W$, $v\in\V$. The elements $w, v$ are called the \textit{components} (or the \textit{projections}) of $p$ with respect to the decomposition $\H^n=\W\cdot\V$ and we will use the notation $w=p_\W,\,v=p_\V$. Let's stress that the components of a point $p\in\H^n$ depend on both the complementary subgroups and also on the order in which they are taken.\end{rem}

Some properties of the projections are described in \cite[Subsection 2.2]{fs2016}. 
For convenience we collect below the ones that we will use in the paper:
\begin{pro}\label{prop_proj}
    Let $\he^n=\W\V$ where $\W$ and $\V$ are complementary subgroups. Let us denote by $P_\W$ and $P_\V$ the \textit{projection maps} onto $\W$ and $\V$ respectively, namely
    \begin{align*} 
    &P_\W:\he^n\to\W &P_\V:\he^n\to\V \\
    &\,P_\W(p)\ceq p_\W &P_\V(p)\ceq p_\V.\,\,
    \end{align*}
    Then\begin{enumerate}
    \item If $\W$ is a normal subgroup, then $P_\V$ is a Lipschitz homomorphism of groups;\\
    Similarly, if $\V$ is normal, then $P_\W$ is a Lipschitz homomorphism of groups.
    \item If $\W$ is a normal subgroup, then the following identities hold:
    \begin{align*}
    &P_\W(p\cdot q)=P_\W(p)\cdot P_\V(p)\cdot P_\W(q)\cdot P_\V(p)^{-1} & P_\V(p\cdot q)=P_\V(p)\cdot P_\V(q)\\
    & P_\W(p^{-1})= P_\V(p)^{-1}\cdot P_\W(p)^{-1}\cdot P_\V(p)
    &P_\V(p^{-1})=P_\V(p)^{-1}.\qquad
    \end{align*}
        \item There exists a constant $\widetilde C=\widetilde C(\W,\V)>0$ such that
        \[\widetilde C(\|P_\W(p)\|+\|P_\V(p)\|)\leq \|p\|\leq \|P_\W(p)\|+\|P_\V(p)\|\quad\text{ for all }p\in\H^n.\]
    \end{enumerate}
\end{pro}

\begin{definition}
Let $\W,\V$ be complementary subgroups of $\H^n$. Given a function $\phi:A\subset\W\to\V$, we define its \textit{intrinsic graph} as the set
    \[\gr_\phi\ceq \{w\cdot\phi(w): w\in A\}.\]
\end{definition}
\begin{definition}\label{def:cone}
   Let $\W, \V$ be complementary subgroups of $\H^n$. If $\beta\geq 0$, we define the \textit{(intrinsic) cone} $C_{\W,\V}(0,\beta)$ of vertex $0$, base $\W$, axis $\V$ and opening $\beta$ as
   \[C_{\W,\V}(0,\beta)\ceq \{p\in\H^n:\|p_\W\|\leq \beta\|p_\V\|\}.\]
Moreover, for every $p\in\H^n \setminus \lbrace 0 \rbrace$ we define the \textit{(intrinsic) cone} of vertex $p$, base $\W$, axis $\V$ and opening $\beta$ as
   \begin{equation}\label{tracone}
   C_{\W,\V}(p,\beta):=p\cdot C_{\W,\V}(0,\beta).
   \end{equation}
For the sake of brevity, when there is no risk of confusion, in the following we will denote such cone as $C_\beta(p)$.
\end{definition}
\begin{rem}
Let $\W,\V$ be complementary subgroups of $\H^n$. For every $p \in \H^n$, $0<\alpha<\beta$ we have $C_\alpha(p)\subset C_{\beta}(p)$. Moreover $\delta_\lambda \big(C_\beta(0)\big)=  C_\beta(0)$ for every $\lambda >0$. Finally $C_0(0)=\V$, while $\overline{\cup_{\beta>0}C_\beta(0)}=\H^n$.
\end{rem}
\begin{definition}\label{def_lipschitz}
    Let $\W,\V$ be complementary subgroups of $\H^n$ and let $\phi:A\se \W\to \V$. We say that $\phi$ is an \textit{intrinsic Lipschitz map} if there exists $M>0$ such that, for all $p\in \gr_\phi$,
    \[C_{\frac{1}{M}}(p)\cap \gr_\phi=\{p\}.\]
    If this is the case, we will say that $\phi$ is intrinsic $M-$Lipschitz. Moreover, for $E\se A$, we define the \textit{Lipschitz constant} of $\phi$ on $E$ as
\[\Lip(\phi,E) \ceq \inf \left\{ M>0 :\,C_{\frac{1}{M}}(p) \cap \gr_{\phi|_E} =\lbrace p \rbrace\text{ for all $p\in \gr_{\phi|_E}$ }  \right\rbrace.\]
\end{definition}
\begin{definition}\label{pointLip}
Let $\W,\V$ be complementary subgroups of $\H^n$ and let $\phi:A\se \W\to \V$. We define the set
\[S_\phi:=\{w\in A: \exists \,\beta> 0 \text{ and } U \text{ open in $A$ with } w\in U, \,C_\beta(w \cdot \phi(w))\cap \,\gr_{\phi|_U}=\{w \cdot \phi(w)\}\}.\]
If $w \in S_\phi$ we say that $\phi$ is \emph{locally intrinsic Lipschitz} at $w$ and we define the \textit{pointwise intrinsic Lipschitz constant} of $\phi$ at $w$ as
    \begin{align*}
  & \lip(\phi,w) \ceq \inf  \left\lbrace M>0:\exists\, r>0 \text{ such that } C_\frac{1}{M}(w \cdot \phi(w)) \cap \gr_{\phi|_{B_\W(w,r)}} =\lbrace w \cdot \phi(w) \rbrace \right\rbrace.
      \end{align*}
      where with $B_\W( w, r)$ we intend, here and in the following, the ball depending on the distance induced on $\W$ by $d$.
\end{definition}

\begin{rem}
   By definition, intrinsic cones depend on the chosen distance $d$. However, the intrinsic Lipschitz continuity property of a given map $\phi$ does not depend on the choice of $d$; also the set $S_\phi$ does not depend on $d$. 
\end{rem}
\begin{defi}
Let $m \geq 0$. Given any left invariant, homogeneous and rotationally invariant distance on $\H^n$ we denote by $\s^m$ the \emph{spherical Hausdorff measure} on $\he^n$ defined for $E\subset \he^n$ by
\[
\s^m(E) \ceq \lim_{r\to0^+} \inf\left\{ \sum_{i\in \N} (2r_i)^m :\exists \;(p_i)_{i}\subset\he^n,\exists\;(r_i)_{i}\text{ with }0<r_i<r\text{ and }E\subset\bigcup_{i\in \N}B(p_i,r_i)  \right\}.
\]
\end{defi}

\begin{defi}
Let $\mu$ be a measure on $\H^n$, $E \se \H^n$ a measurable set for $\mu$ and $p \in E$. We say that $p$ is a \emph{point of density} of $E$ with respect to $\mu$ if
\[
\lim_{r \to 0^+}\frac{\mu(E \cap B(p,r))}{\mu(B(p,r))}=1.
\]

\end{defi}

\begin{teo}[{\cite[Theorem 3.9]{fs2016}}]\label{teo_grahlreg}
    Let  $\W,\V$ be complementary subgroups of $\H^n$ and let $\phi:\W \to \V$ be an intrinsic Lipschitz map. Denote by $k$ the (metric) dimension of $V$. Then $\s^{Q-k} \res \gr_\phi$ is $(Q-k)$-Ahlfors regular on $\gr_{\phi}$.
\end{teo}

\begin{defi}
Let $\W,\V$ be complementary subgroups of $\H^n$ and  $\phi:A \se \W \to \V$. We define the \emph{graph distance} $\rho_\phi:A \times A \to [0,+\infty)$ as
\[
\rho_\phi(w_1,w_2) \ceq  \frac{1}{2}\left( \|(p_1^{-1}\cdot p_2)_\W\|+\|(p_2^{-1}\cdot p_1)_\W\| \right)\quad \forall w_1,w_2 \in A
\]
where $p_i=w_i \cdot \phi(w_i)$ for $i=1,2$.
\end{defi}

\begin{pro}[{\cite[Proposition 4.59]{sc2016} and \cite[Remark 3.6]{fs2016}}]\label{pro_quasidist}
Let $\W,\V$ be complementary subgroups of $\H^n$ and  $\phi:A \se \W \to \V$ be intrinsic $M$-Lipschitz. Then the graph distance $\rho_\phi$ is a quasi-distance, i.e.,
\begin{itemize}
\item[(i)] $\rho_\phi(w_1,w_2)=\rho_\phi(w_2,w_1)$ for every $w_1,w_2 \in A$,
\item[(ii)]$\rho_\phi(w_1,w_2)=0$ if and only if $w_1=w_2$,
\item[(iii)] there exists a constant $C>1$ such that 
\[
\rho_\phi(w_1,w_2) \leq C(\rho_\phi(w_1,w_3)+\rho_\phi(w_3,w_2)) \text{ for every }w_1,w_2,w_3 \in A.
\]
\end{itemize}
Moreover, there exists a positive constant $C=C(M)>1$ such that
\[
\frac{1}{C}\rho_\phi(w_1,w_2) \leq d(w_1 \cdot \phi(w_1),w_2 \cdot \phi(w_2)) \leq  C \rho_\phi(w_1,w_2)
\]
for every $w_1,w_2 \in A$. This means that $\rho_\phi$ is equivalent to the distance on the graph of $\phi$.
\end{pro}

\begin{rem}
If $\W$ is a normal subgroup of $\he^n$, then Definition \ref{pointLip} is equivalent to the following one, which generalizes \eqref{def_sf}:
\[S_\phi=\left\{w\in A: \limsup_{\W\,\ni\, y\to w}\frac{d(\phi(y),\phi(w))}{\rho_\phi(y,w)}<+\infty\right\}.\]
In fact, on one side, assume that there exists $\beta>0$ and $U$ open in $A$ with $w\in U$ and 
\[C_\beta(w\cdot\phi(w))\cap \gr_{\phi|U}=\{w\cdot\phi(w)\}.\]
Then, for every $y\in U$ with $y\neq w$, it holds $y\cdot\phi(y)\not\in C_\beta(w\cdot \phi(w))$. Denoting by $p:=w\cdot \phi(w)$ and $q:=y\cdot \phi(y)$, by definition of cone, we get that $p^{-1}\cdot q\not \in C_\beta (0)$ and so
\[\|(p^{-1} \cdot q)_\W\|>\beta \|(p^{-1} \cdot q)_\V\|=\beta\|\phi(w)^{-1}\cdot \phi(y)\|,\]
where the last identity follows by the fact that $P_\V$ is a group homomorphism (see Proposition \ref{prop_proj}). Hence
\[d(\phi(y),\phi(w))=\|\phi(w)^{-1}\cdot \phi(y)\|< \frac{1}{\beta}\|(p^{-1}\cdot q)_\W\|\leq \frac{1}{\beta}\Big[\|(p^{-1}\cdot q)_\W\|+\|(q^{-1}\cdot p)_\W\|\Big]=\frac{2}{\beta}\rho_\phi(y,w)
\] for every $y\in U$ with $y\neq w.$ It follows
that 
\[\limsup_{\W\,\ni\, y\to w} \,\frac{d(\phi(y),\phi(w))}{\rho_\phi(y,w)}\leq \frac{2}{\beta}<+\infty.\]
On the other hand, assume that 
\[\limsup_{\W\,\ni\, y\to w} \,\frac{d(\phi(y),\phi(w))}{\rho_\phi(y,w)}<+\infty,\]
which implies that there exists $L>0$ and $U\ni w$ open in $A$ such that, for every $y\in U$,
\begin{equation}\label{defoflimsup}\|\phi^{-1}(w)\cdot \phi(y)\|=d(\phi(y),\phi(w))\leq L\rho_\phi(y,w).\end{equation} For every $y\in U$, let us denote as before $q:=y\cdot \phi(y)$ and $p:=w\cdot\phi(w)$.\\
\textbf{Claim:} $\exists\, C>0$ (not depending on $y)$ such that $\rho_\phi(w,y)\leq C\|(p^{-1}\cdot q)_\W\|$ for every $y\in U$.\\
Assuming the Claim, thanks to \eqref{defoflimsup}, we get that, for a suitable constant $\tilde L$ and for every $y\in U$,
\[\|(p^{-1}\cdot q)_\V\|=\|\phi(w)^{-1}\cdot \phi(y)\|\leq \tilde L\|(p^{-1}\cdot q)_\W\|,\]
which means that the point $p^{-1}q$ does not belong to the cone $C_{\frac{1}{\tilde L}}(0)$. Therefore $q\not\in C_{\frac{1}{\tilde L}}(p)$ for every $y\in U$ and 
\[C_{\frac{1}{\tilde L}}(p)\cap \gr_{\phi|U}=\{p\}\] concluding the proof of the equivalence.\\
It remains to show the validity of the Claim: such a conclusion can be obtained proceeding as in the second part of the proof of Theorem 4.60 in \cite{sc2016}. In particular one get that, for every $\epsilon>0$ there exists a constant $\overline C=\overline C(\epsilon)$ such that for every $y\in A$
\[\rho_\phi(w,y)\leq \overline C(\epsilon)\|(p^{-1}\cdot q)_\W\|+\epsilon \|(p^{-1}\cdot q)_\V\|.\]
From \ref{defoflimsup}, if $y\in U$ we get
\[\rho_\phi(w,y)\leq \overline C(\epsilon)\|(p^{-1}\cdot 
 q)_\W\|+\epsilon L\rho_\phi(w,y).\]
Fixing $\epsilon < 1/L$, we finally get
\[\rho_\phi(w,y)\leq \frac{\overline C}{1-\epsilon L}\|(p^{-1}\cdot q)_\W\|\]
for every $y\in U$, proving the Claim.
\end{rem}

\begin{definition}
Let $\W,\V$ be complementary subgroups of $\H^n$ and $\phi:A \se \W \to \V$. We say that $\phi$ is an \textit{intrinsic linear} map if its graph $\gr_\phi$ is a homogeneous subgroup of $\H^n$.
\end{definition}
\begin{rem}
Another characterization of the notion of intrinsic linear map can be given as follows, see \cite{vittone2021}. For every $w\in\mathbb{W}$ define $w_H\in\mathbb{R}^{2n+1-k}$ as
\begin{align*}
   & w_H:=(x_{k+1},\ldots, y_n)\quad \mbox{if}\ k<n\quad \mbox{and}\quad  w=(x_{k+1},\ldots, y_n,t),\\
   & w_H:=(y_1,\ldots, y_n)\quad\quad \mbox{if}\ k=n\quad \mbox{and}\quad  w=(y_1,\ldots, y_n,t).
\end{align*}
Then, $\phi$ is intrinsic linear if and only if there exists a $k\times (2n-k)$ matrix $M$ such that, for every $w\in\mathbb{W}$, $\phi(w)=Mw_H$ (identifying $M$ with a linear map $M:\R^{2n-k}\to\R^k\equiv \V$).
\end{rem}

\begin{defi}\label{intdiff} 
Let $\W,\V$ be complementary subgroups of $\H^n$ and $\phi:A \se \W \to \V$ where $A$ is a relatively open set. We say that $\phi$ is \textit{intrinsically differentiable} at $ \bar w \in A$ if there exists an intrinsic linear map $d\phi_{\bar w}:\W\to\V$ such that
\[\lim_{\W\,\ni\,w\to 0}\frac{d(\phi_{\bar w}(w),d\phi_{\bar w}(w))}{\|w\|}=0,\]
    where $\phi_{\bar w}$ is the map whose intrinsic graph is given by $(\bar w\cdot\phi(\bar w))^{-1} \cdot \gr_\phi$.
    The map $d\phi_{\bar w}$ is called the \textit{intrinsic differential} of $\phi$ at $\bar w$.\end{defi}
    \begin{rem}
If $\W$ is a normal subgroup, then the explicit expression of $\phi_{\bar w}$ is given by
    \begin{equation}\label{Wnorm}\phi_{\bar w}(w)\ceq \phi(\bar w)^{-1} \cdot \phi(\bar w \cdot \phi(\bar w)\cdot w \cdot \phi(\bar w)^{-1}).\end{equation}
If instead $\V$ is normal, then $\phi_{\bar w}$ can be written as follows
\begin{equation}\label{Vnorm}\phi_{\bar w}(w)\ceq w^{-1}\cdot\phi(\bar w)^{-1}\cdot w\cdot \phi(\bar w\cdot w).\end{equation} 
A general expression for the map $\phi_{\bar w}$, which has as special cases formulas \eqref{Wnorm} and \eqref{Vnorm}, can be found for instance in \cite[Proposition 2.21]{fs2016}.
\end{rem}
In the proof of the Stepanov differentiability theorem we will need two main results: an equivalent characterization of intrinsic differentiability and the Rademacher Theorem for intrinsic graphs. We now state both.
\begin{teo}[{\cite[Theorem 4.15]{fssc2010}, see also \cite[Theorem 3.2.8]{fms2014}}]\label{teo_fsscequivcar}
Fix $1 \leq k \leq n$ and let $\W,\V$ be complementary subgroups of $\H^n$ where $\V$ is a horizontal subgroup of dimension $k$.  Let $A \se \W$ be an open set and $\phi:A \to \V$. Fix $\bar w \in A$ and define $\bar p \ceq \bar w \cdot \phi (\bar w)$. Then the following statements are equivalent:
\begin{enumerate}
\item[(i)] $\phi$ is intrinsically differentiable at $\bar w$;
\item[(ii)] there exists a vertical subgroup $\T_{\phi,\bar p}$, complementary to $\V$, such that for every $\alpha>0$ there exists $\bar r=\bar r(\phi,\bar w,\alpha)>0$ such that
\[
C_{\T_{\phi,\bar p},\V}(\bar p,\alpha) \cap \gr_{\phi|_{B_\W(\bar w,\bar r)}}=\lbrace \bar p \rbrace.
\]
\end{enumerate}
\end{teo}
\begin{teo}[{\cite[Theorem 1.1 and Theorem 1.5]{vittone2021}}]\label{teo_RadVit}
Fix $1 \leq k \leq n$ and let $\W,\V$ be complementary subgroups of $\H^n$ where $\V$ is a horizontal subgroup of dimension $k$.  Let $A \se \W$ be a set and $\phi:A \to \V$ be intrinsic Lipschitz. Then there exists an intrinsic Lipschitz map $\tilde{\phi}:\W \to \V$ such that $\tilde{\phi}|_A \equiv \phi$ and $\tilde{\phi}$ is intrinsically differentiable almost everywhere on $\W$.
\end{teo}
\section{Proof of Stepanov Theorem}

\subsection{Proof of Stepanov Theorem for intrinsic graphs of arbitrary codimension}
In this subsection, we provide the proof of Stepanov differentiability theorem starting with the case of maps from horizontal subgroups to normal subgroups and then our main result which involves maps going from normal subgroups to horizontal subgroups.\\
The first case is a fairly easy consequence of the  ``classical" Stepanov theorem for maps between Carnot groups \cite{vodopyanov2000}. Indeed, by applying \cite[Proposition 3.25]{as2009},  we can reduce the intrinsic differentiability of a map to the Pansu differentiability of its graph map.
\begin{teo}\label{stepanov1}
Fix $1 \leq k \leq n$ and let $\W,\V$ be complementary subgroups of $\H^n$ where $\V$ is a horizontal subgroup of dimension $k$. Let $A \se \V$ be an open set and $\phi:A \to \W$. Then $\phi$ is intrinsically differentiable almost everywhere on $S_\phi$.
\end{teo}
\begin{proof}
  Let $\phi:A\subset \V\to\W$ be as in the assumptions.
  Define $\Phi_\phi:A\to \h^n$ to be the graph map:
  \[\Phi_\phi (v)\ceq v\cdot\phi(v).\]
  We start by proving the following elementary fact: $\bar v\in S_\phi$ if and only if $\Phi_\phi$ is pointwise (metric) Lipschitz continuous at $\bar v$, that is
  \begin{equation}\label{metriclip}\limsup_{A\,\ni\, v\,\to\,\bar v}\frac{\|(\Phi_\phi(\bar v))^{-1}\cdot \Phi_\phi(v)\|}{\|\bar v^{-1}\cdot v\|}<+\infty.\end{equation}
  Suppose first $\bar v\in S_\phi$. This means that there exists $\beta>0$ and $U\subset A$ open containing $\bar v$ such that
  \[C_{\V,\W}(\Phi_\phi(\bar v), \beta)\cap \gr_{\phi|_{U}}=\{\Phi_\phi(\bar v)\}.\]
  Equivalently, after a left translation and \eqref{tracone},
  \[C_{\V,\W}(e, \beta)\cap (\Phi_\phi(\bar v))^{-1}\cdot \gr_{\phi|_U}=\{e\}.\]
  For notational simplicity, let us write $\bar p:=\Phi_\phi(\bar v)$.
  By Definition \ref{def:cone}, we get that, for every point $p=v\cdot \phi(v)\in \gr_{\phi|_{U}}$, 
  \[\|(\bar p^{-1}\cdot p)_\W\|\leq \frac{1}{\beta}\|(\bar p^{-1}\cdot p)_\V\|.\]
Since $\W$ is normal, the projection on $\V$ is a group homomorphism (Proposition \ref{prop_proj}): hence $(\bar p^{-1}\cdot p)_\V=\bar v^{-1}\cdot v$.
  Therefore we get 
  \[\|\bar p^{-1}\cdot p\|\leq \|(\bar p^{-1}\cdot p)_\W\|+\|(\bar p^{-1}\cdot p)_\V\|\leq (1+\frac{1}{\beta})\|\bar v^{-1}\cdot{v}\|,\]
  for every $v\in U$, which implies \eqref{metriclip}.\\
 On the other hand, assuming \eqref{metriclip} and keeping the same notation as before, there exists $U\subset A$ open with $\bar v\in U$ and $L>0$ such that 
 \[\|\bar p^{-1}\cdot p\|\leq L\|\bar v^{-1}\cdot v\|=L\|(\bar p^{-1}\cdot p)_\V\|\]
 for all $v\in U$. By Proposition $\ref{prop_proj}$, there exists a constant $\widetilde C>0$ such that $\|\bar p^{-1}\cdot p\|\geq \widetilde C\|(\bar p^{-1}\cdot p)_\W\|$. We deduce that
 \[\|(\bar p^{-1}\cdot p)_\W\|\leq \frac{L}{\widetilde C} \|(\bar p^{-1}\cdot p)_\V\|\]
 which implies, arguing as in the first part of the proof, that $\bar v\in S_\phi$.\\
 We move now to the proof of the Theorem: since $\V$ is horizontal, we can identify  $\V\equiv \R^k$  for some $1\leq k\leq n$. Hence, by applying the ``classical" Stepanov differentiability theorem for maps between Carnot groups (see \cite[Theorem 3.1]{vodopyanov2000}), we deduce that the graph map $\Phi_\phi:A\subset \V\equiv \R^k\to \h^n$ is Pansu-differentiable almost everywhere in the set of points where $\Phi_\phi$ is pointwise Lipschitz continuous, which coincides by the previous argument with $S_\phi$. Then we conclude by applying \cite[Proposition 3.25]{as2009}: a map $\phi:A\subset \V\to\W$ is intrinsically differentiable at $\bar v$ if and only if the graph map $\Phi_\phi: A\subset \V\to \h^n$ is Pansu-differentiable at $\bar v$.
\end{proof}
\begin{rem}
 We point out that all the results used in the proof of Theorem \ref{stepanov1} hold in general Carnot groups $\G$ which can be written as $\G=\V\W$ where $\V$ is a horizontal subgroup and $\W$ is normal. Therefore,  Theorem \ref{stepanov1} holds even in the generality described above for maps $\phi:A\subseteq\V\to\W$. 
\end{rem}
Let us now move to the main result of this paper, concerning the proof of Stepanov differentiability theorem for maps from a normal subgroup to an abelian one. Theorem \ref{maindiff} below combined with Theorem \ref{stepanov1} completes the proof of the Stepanov differentiability theorem.

\begin{teo}\label{maindiff}
Fix $1 \leq k \leq n$ and let $\W,\V$ be complementary subgroups of $\H^n$ where $\V$ is a horizontal subgroup of dimension $k$. Let $A \se \W$ be an open set and $\phi:A \to \V$. Then $\phi$ is intrinsically differentiable almost everywhere on $S_\phi$.
\end{teo}
\begin{proof}
For convenience we split the proof into several steps.\\

\emph{Step 1: Split $S_\phi$ into countably many sets where $\phi$ is intrinsic Lipschitz.}\\
For $j \in \N$ we define
\begin{equation}\label{defofCj}
C_j \ceq \left\lbrace w \in A: C_{1/j}(w \cdot \phi(w)) \cap \gr_{\phi|_{B_\W(w,1/j)}}=\lbrace w \cdot \phi(w)\rbrace \right\rbrace.
\end{equation}
Then each $C_j$ is measurable and it is clear that $S_\phi=\bigcup_{j \in \N}C_j$. Then we express each $C_j$ as the union of measurable sets $C_{j,1},C_{j,2},\dots$ such that $\diam(C_{j,i})<\tfrac 1j$ for every $i,j \in \N$. We can do that for example intersecting $C_j$ with countably many balls of diameter smaller than $\tfrac 1j$. Then we have that $S_\phi=\bigcup_{j,i \in \N}C_{j,i}$ and $\phi|_{C_{j,i}}$ is intrinsic Lipschitz.\\

\emph{Step 2: Use Theorem \ref{teo_RadVit} to extend each $\phi|_{C_{j,i}}$.}\\
By Theorem \ref{teo_RadVit} for every $j,i \in \N$ there exists an intrinsic Lipschitz map $\tilde{\phi}:\W \to \V$ such that $\tilde{\phi}|_{C_{j,i}}\equiv \phi|_{C_{j,i}}$ and $\tilde{\phi}$ is intrinsically differentiable almost everywhere on $\W$. Fix  $\bar w \in C_{j,i}$ such that $\bar w$ is a point of intrinsic differentiability for $\tilde{\phi}$ and $\bar w \cdot \phi(\bar w)$ is a density point of $\gr_{\phi|C_{j,i}}$ with respect to $\s^{Q-k}\res \gr_{\tilde{\phi}}$ (recall that from Theorem \ref{teo_grahlreg} $\s^{Q-k}\res \gr_{\tilde{\phi}}$ is a $(Q-k)$-Ahlfors regular measure on $\gr_{\tilde \phi}$, so that Lebesgue density theorem holds). By \cite[Remark 4.6]{vittone2021}, there exists a constant $\overline C>0$ such that, denoting by $\Phi$ the graph map $\Phi(w):=w\cdot \tilde\phi(w)$,
\[\overline C^{-1}\s^{Q-k}\res\gr_{\tilde \phi}\leq \Phi_{\#}(\leb^{2n+1-k}\res\W)   \leq \overline C\s^{Q-k}\res \gr_{\tilde \phi}.\]
This implies that the set of points $\bar w$ with the previous properties is a full-measure set in $ C_{j,i}$.\\
If we prove that $\bar w$ is also a point of intrinsic differentiability for $\phi$ then we are done.\\

\emph{Step 3: Without loss of generality, one  can assume  $\bar w=0$ and $\phi(\bar w)=0$.}\\
Assuming that $\bar w=0$ and $\phi(\bar w)=0$ is equivalent to replace the function $\phi$ with the translated function $\phi_{\bar w}$ (see \eqref{Wnorm}) since $\phi_{\bar w}(0)=0$. Notice that, by definition, $\phi$ is intrinsically differentiable at $\bar w$ if and only if $\phi_{\bar w}$ is intrinsically differentiable at $0$. Hence, it suffices to show that all the properties we will use on the map $\phi$ are true also for $\phi_{\bar w}$. Again, the differentiability of 
$\tilde \phi$ is preserved by translation of the graph. The same holds for the intrinsic Lipschitz property of $\tilde \phi$. Moreover $\bar w\cdot \phi(\bar w)$ is a density point of $\gr_{\phi|C_{j,i}}$ if and only if $0$ is a density point of $(\bar w\cdot \phi(\bar w))^{-1}\cdot\gr_{\phi|C_{j,i}}$ (by invariance of the distance and the measure). The last condition to be verified is to show that there exists $\delta>0$ such that for all $w \in C_{j,i}$ one has
\begin{equation}\label{eq_tranfun}
B_\W(\phi(\bar w)^{-1} \cdot \bar w^{-1}\cdot w \cdot \phi (\bar w),\delta) \se \phi(\bar w)^{-1} \cdot \bar w^{-1} \cdot B_\W(w,\tfrac 1j) \cdot \phi(\bar w).
\end{equation}
In fact, if \eqref{eq_tranfun} holds, then we obtain \emph{Step 3} upon noticing that the set corresponding to $B_\W(w,\frac{1}{j})$ after the translation of the graph of $\phi$ is exactly $\phi(\bar w)^{-1} \cdot \bar w^{-1} \cdot B_\W(w,\tfrac 1j) \cdot \phi(\bar w)$. Hence, property \eqref{defofCj} remains true for the function $\phi_{\bar w}$ replacing $\frac{1}{j}$ with $\delta$.
\\We are left to prove \eqref{eq_tranfun}.  Since 
\[
\lim_{\W \ni a \to 0}\|\phi(\bar w) \cdot a \cdot \phi(\bar w)^{-1}\|=0,
\]
there exists $\delta>0$ such that, if $\|a\|<\delta$, then 
\begin{equation}\label{eq_aa1}
\|\phi(\bar w) \cdot a \cdot \phi(\bar w)^{-1}\|<\tfrac 1j.
\end{equation}
We define $\tilde{w} \ceq \phi(\bar w)^{-1}\cdot \bar w^{-1}\cdot  w \cdot \phi(\bar w)$ and we claim that 
\begin{equation}\label{eq_aa2}
\phi(\bar w) \cdot B_\W(\tilde{w},\delta) \cdot \phi(\bar w)^{-1}\se B_\W (\phi(\bar w) \cdot \tilde{w}\cdot \phi(\bar w)^{-1},\tfrac 1j).
\end{equation}
If \eqref{eq_aa2} holds, then
\begin{align*}
B_W(\tilde{w},\delta) &\se \phi(\bar w)^{-1}\cdot B_\W(\phi(\bar w) \cdot \tilde{w} \cdot \phi(\bar w)^{-1},\tfrac 1j) \cdot \phi(\bar w)\\
&=\phi(\bar w)^{-1}B_\W( \bar w^{-1}\cdot w, \tfrac 1j) \cdot \phi(\bar w)\\
&=\phi(\bar w)^{-1}\cdot \bar w^{-1}\cdot B_\W (w,\tfrac 1j) \cdot \phi(\bar w).
\end{align*}
proving \eqref{eq_tranfun}. We are left to prove \eqref{eq_aa2}. Let $y \in \phi(\bar w) \cdot B_\W (\tilde{w},\delta) \cdot \phi(\bar w)^{-1}$. Then $y=\phi(\bar w) \cdot x \cdot \phi(\bar w)^{-1}$ for some $x \in B_\W(\tilde{w},\delta)$. Since $d(x,\tilde{w})<\delta$, by \eqref{eq_aa1} we have $\|\phi(\bar w) \cdot \tilde{w}^{-1}\cdot x \cdot \phi(\bar w)^{-1}\|<\tfrac 1j$. The latter implies that
\[
d(y,\phi(\bar{w}) \cdot \tilde{w} \cdot \phi(\bar w)^{-1})=\| \phi(\bar w) \cdot \tilde{w}^{-1}\cdot \phi(\bar w)^{-1}\cdot y \|=\|\phi(\bar w)\cdot \tilde{w}^{-1} \cdot x \cdot \phi(\bar w)^{-1}\|<\tfrac 1j,
\]
finally proving \eqref{eq_aa2}.\\

\emph{Step 4: Use the equivalent characterization from Theorem \ref{teo_fsscequivcar}}.\\
Since $\tilde{\phi}$ is intrinsically differentiable at $0$, by Theorem \ref{teo_fsscequivcar}, there exists a vertical subgroup $\T_{\tilde{\phi},0}$ such that for every $\alpha>0$ there exists $\tilde r=\tilde r(\tilde{\phi},0,\alpha)>0$ such that
\begin{equation}\label{eq_diffphitil}
C_{\T_{\tilde{\phi},0},\V}(0,\alpha) \cap \gr_{\tilde{\phi}|_{B_\W(0,\tilde r)}}=\lbrace 0 \rbrace. 
\end{equation}
Using again Theorem \ref{teo_fsscequivcar}, if we show that for every $\alpha>0$ there exists $\bar r=\bar r(\phi,0,\alpha)>0$ such that 
\[
C_{\T_{\tilde{\phi},0},\V}(0,\alpha) \cap \gr_{\phi|_{B_\W(0,\bar r)}}=\lbrace 0 \rbrace,
\]
then we get that $\phi$ is intrinsically differentiable at $0$. Assume not: then there exists a certain $\alpha>0$ such that for every $\bar r>0$ one has
\begin{equation}\label{eq_contr}
C_{\T,\V}(0,\alpha) \cap \gr_{\phi|_{B_\W(0,\bar r)}} \neq \lbrace 0 \rbrace.
\end{equation}
where, for the sake of brevity, we write $\T \ceq \T_{\tilde{\phi},0}$. From \eqref{eq_contr} we obtain that there is a sequence of points $(x_h)_{h \in \N} \se \W$ such that $x_h \xrightarrow{h \to \infty}0$ and 
\begin{equation}\label{eq_cond1}
p_h \ceq x_h  \cdot \phi(x_h) \in C_{\T,\V}(0,\alpha).
\end{equation} 
On the other hand $\tilde{\phi}$ is intrinsically differentiable at $0$ so, by \eqref{eq_diffphitil}, for $h$ sufficiently large 
\begin{equation}\label{eq_cond2}
\widetilde{p_h}\ceq x_h \cdot \tilde{\phi}(x_h) \not \in C_{\T,\V}(0,2\alpha).
\end{equation}
\phantom{}

\emph{Step 5: Prove that $(p_h)_\T=(\widetilde{p_h})_\T$.}\\
Here and in the following of the proof we will use the following notations: in order to indicate the components of a point $q \in \H^n$ with respect to the splitting $\H^n=\W \cdot \V$ we will use $q=q_\W \cdot  q_\V$; in order to indicate the components of a point $q \in \H^n$ with respect to the splitting $\H^n=\T \cdot \V$ we will use $q=q_\T \cdot q_\V^\T$. Notice that, in general, $q_\V \neq q_\V^\T$.
We observe that 
\[
(p_h)_\W \cdot (p_h)_\V=p_h=(p_h)_\T \cdot (p_h)_\V^\T=((p_h)_\T)_\W \cdot \underbrace{((p_h)_\T)_\V \cdot (p_h)_\V^\T}_{\in \V}.
\]
By the uniqueness of the components (see Remark \ref{rem_componenti}) we conclude that  $(p_h)_\W=((p_h)_\T)_\W$ and so $x_h=((p_h)_\T)_\W$. In the same fashion we obtain that $x_h=((\widetilde{p_h})_\T)_\W$ and so $((p_h)_\T)_\W=((\widetilde{p_h})_\T)_\W$. Now we observe that 
\[
\T \ni ((p_h)_\T)^{-1}\cdot (\widetilde{p_h})_\T=(((p_h)_\T)_\V)^{-1} \cdot  \underbrace{(((p_h)_\T)_\W)^{-1} \cdot ((\widetilde{p_h})_\T)_\W}_{=0} \cdot ((\widetilde{p_h})_\T)_\V \in \V
\]
but $\T$ and $\V$ are complementary so $\T \cap \V=\lbrace 0 \rbrace$ implying that $(p_h)_\T=(\widetilde{p_h})_\T$.\\

\emph{Step 6: Prove that $d(\phi(x_h),\tilde{\phi}(x_h)) \geq K\|x_h\|$ for some $K>0$.}\\
We observe preliminarily that, by the fact that $((p_h)_\T)_\W=x_h$, 
\begin{equation}\label{eq_aux1}
(p_h)_\V^\T=((p_h)_\T)^{-1} \cdot p_h=(((p_h)_\T)_\V)^{-1}\cdot x_h^{-1}\cdot x_h \cdot \phi(x_h)=(((p_h)_\T)_\V)^{-1}\cdot \phi(x_h)
\end{equation}
and, in the same fashion, since $(p_h)_\T=(\widetilde p_h)_\T$,
\begin{equation}\label{eq_aux2}
(\widetilde{p_h})_\V^\T=(((\widetilde p_h)_\T)_\V)^{-1}\cdot \tilde{\phi}(x_h) = (((p_h)_\T)_\V)^{-1}\cdot \tilde{\phi}(x_h) .
\end{equation}
From \eqref{eq_cond1} and \eqref{eq_cond2} we obtain, for $h$ sufficiently large,
\begin{equation}\label{eq_cond3}
\begin{cases}
\| (p_h)_\T\| \leq \alpha \|(p_h)_\V^\T\|,\\
\| (\widetilde{p_h})_\T\| >2\alpha \|(\widetilde{p_h})_\V^\T\|.
\end{cases}
\end{equation}
By the left-invariance of the distance we have 
\begin{align*}
d(\phi(x_h),\tilde{\phi}(x_h))&=d\left((((p_h)_\T)_\V)^{-1} \cdot \phi(x_h),((({p_h})_\T)_\V)^{-1} \cdot \tilde{\phi}(x_h)\right)\\
&=d\left((p_h)_\V^\T,(\widetilde{p_h})_\V^\T\right)\\
&=\|((\widetilde{p_h})_\V^\T)^{-1}\cdot (p_h)_\V^\T \|\\
&\geq \|(p_h)_\V^\T\|- \|(\widetilde{p_h})_\V^\T\|\\
&\geq \frac{1}{2\alpha}\|(p_h)_\T\|.
\end{align*} 
where we used, in order from the first to the last line, the left-invariance of the distance, \eqref{eq_aux1} and \eqref{eq_aux2}, the definition of norm associated to the distance, the reverse triangular inequality and, finally, \eqref{eq_cond3}. From Proposition \ref{prop_proj}, we obtain that 
\begin{equation}\label{eq_auxfin}
\|(p_h)_\T\|=\| x_h \cdot ((p_h)_\T)_\V\| \geq \widetilde{C}\|x_h\|
\end{equation}
where $\tilde{C}$ is a constant only depending on $\W$ and $\V$. The latter proves that $d(\phi(x_h),\tilde{\phi}(x_h)) \geq K\|x_h\|$ for some $K>0$.\\

\emph{Step 7: Construction of an auxiliary sequence $(q_h)_{h \in \N}$ with certain properties.}\\
Since we are assuming that $0$ is a density point of $\gr_{\phi|_{C_{j,i}}}$ in $\gr_{\tilde{\phi}}$, by \cite[pag. 409]{brz2004} there exists a sequence $(q_h)_{h \in \N} \se \gr_{\phi|_{C_{j,i}}}$ such that
\begin{equation*}
d(q_h,\widetilde{p_h})=o(d(0,\widetilde{p_h}))=o(\|\widetilde{p_h}\|).
\end{equation*}
In other words, for every $\ve>0$ there exists $\bar h \in \N$ such that for every $h> \bar h$ 
\begin{equation}\label{eq_aux11}
d(q_h,\widetilde{p_h}) \leq \ve \|\widetilde{p_h}\|.
\end{equation}
We observe that 
\begin{equation}\label{eq_aux12}
\|\widetilde{p_h}\|=\|x_h \cdot \tilde{\phi}(x_h)\| \leq \|x_h\|+\|\tilde{\phi}(x_h)\|
\end{equation}
and since $\tilde{\phi}$ is intrinsically Lipschitz and $\tilde{\phi}(0)=\phi(0)=0$ we obtain $\| \tilde{\phi}(x_h)\| \leq L\|x_h\|$ for some constant $L>0$. The latter together with \eqref{eq_aux11} and \eqref{eq_aux12} implies that for every $\ve>0$ there exists  $\bar h_1 \in \N$ such that for every $h> \bar h_1$ 
\begin{equation}\label{eq_aux13}
d(q_h,\widetilde{p_h})\leq \ve \|x_h\|.
\end{equation}
Moreover, for every $h \in \N$ there exists  $y_h \in C_{j,i}$ such that 
\[
q_h=y_h \cdot \phi(y_h)=y_h \cdot \tilde{\phi}(y_h),
\]
so that we can rewrite \eqref{eq_aux13} as
\begin{equation}\label{eq_aux14}
d(y_h \cdot \tilde{\phi}(y_h),x_h \cdot \tilde{\phi}(x_h)) \leq \ve \|x_h\|.
\end{equation}
Since $\W$ is normal, the projection on $\V$ is Lipschitz continuous (again Proposition \ref{prop_proj}) and there exists a constant $D>0$ such that
\[
d(\tilde{\phi}(y_h),\tilde{\phi}(x_h))\leq D d(y_h \cdot \tilde{\phi}(y_h),x_h \cdot \tilde{\phi}(x_h)).
\]
The latter together with \eqref{eq_aux14} implies that for every $\ve>0$ there exists  $\bar h_2 \in \N$ such that for every $h> \bar h_2$ 
\begin{equation}\label{eq_aux15}
d(\tilde{\phi}(y_h),\tilde{\phi}(x_h)) \leq \ve \|x_h\|.
\end{equation} 
Moreover, since $y_h \in C_{j,i}\se C_j$ it follows from \eqref{defofCj} and \eqref{eq_tranfun} that 
\begin{equation}\label{eq_aux21}
C_{1/j}(q_h) \cap \gr_{\phi|_{B_\W(y_h,\delta)}}=\lbrace q_h \rbrace,
\end{equation}
where $\delta>0$ is found as in \eqref{eq_tranfun}. 
\\

\emph{Step 8: Prove that $d(x_h,y_h)\xrightarrow{h \to +\infty}0$.}\\
By \eqref{eq_aux13} and the fact that $x_h\to 0$, we know that $d(q_h,\tilde{p_h}) \xrightarrow{h \to +\infty}0$. If $\widetilde C$ is as in \eqref{eq_auxfin}, by Proposition \ref{prop_proj} we have 
\[
\widetilde C\|\tilde{\phi}(x_h)^{-1}\cdot x_h^{-1} \cdot y_h \cdot \tilde{\phi}(x_h) \|=\widetilde C\| (\widetilde{p_h}^{-1}\cdot q_h)_\W \| \leq \|\widetilde{p_h}^{-1}\cdot q_h\|=d(q_h,\widetilde{p_h})
\]
which implies
\begin{equation}\label{eq_aux41}
\|\tilde{\phi}(x_h)^{-1}\cdot x_h^{-1} \cdot y_h \cdot \tilde{\phi}(x_h) \| \xrightarrow{h \to +\infty}0.
\end{equation}
Recall that $\tilde{\phi}$ is intrinsic Lipschitz, therefore continuous, so $x_h \xrightarrow{h \to +\infty}0$ implies $\tilde\phi(x_h)\to\tilde \phi(0)=\phi(0)=0$. We conclude that $d(x_h,y_h) \xrightarrow{h \to +\infty}0$  upon observing that from \eqref{eq_aux41} we get
\[
d(x_h,y_h)=\|x_h^{-1}\cdot y_h\|=\|\tilde{\phi}(x_h) \cdot \tilde{\phi}(x_h)^{-1}\cdot x_h^{-1} \cdot y_h \cdot \tilde{\phi}(x_h) \cdot \tilde{\phi}(x_h)^{-1}\| \xrightarrow{h \to +\infty}0.
\] 
The latter implies that, for $h$ sufficiently large, $x_h \in B_\W(y_h,\delta)$.\\

\emph{Step 9: Conclude the proof obtaining a contradiction.}\\
Because of \emph{Step 8} and \eqref{eq_aux21}, we have for $h$ large enough
\begin{equation}\label{eq_aux31}
p_h=x_h \cdot \phi(x_h) \not \in C_{1/j}(q_h)=q_h \cdot C_{1/j}(0).
\end{equation}
The latter is true unless $q_h=p_h$, but in that case we would get $\phi(x_h)=\tilde{\phi}(x_h)$ and so we would obtain a contradiction with \textit{Step 6}. From \eqref{eq_aux31} we obtain
\begin{equation}\label{eq_aux32}
q_{h}^{-1}\cdot p_h \not \in C_{1/j}(0) \Rightarrow \| (q_{h}^{-1}\cdot p_h)_\W \|>\frac{1}{j} \| (q_{h}^{-1}\cdot p_h)_\V \|.
\end{equation}
From an explicit computation of the projections (see Proposition \ref{prop_proj}), 
we get
\begin{equation}
\begin{cases}\label{eq_aux33}
(q_{h}^{-1}\cdot p_h)_\V = \tilde{\phi}(y_h)^{-1}\cdot \phi(x_h),\\
(q_{h}^{-1}\cdot p_h)_\W=\tilde{\phi}(y_h)^{-1}\cdot y_h^{-1}\cdot x_h \cdot \tilde{\phi}(y_h).
\end{cases}
\end{equation}
From \eqref{eq_aux32} and \eqref{eq_aux33} we obtain
\[
\| \tilde{\phi}(y_h)^{-1} \cdot \phi(x_h)\|<j\| \tilde{\phi}(y_{h})^{-1}\cdot y_h^{-1}\cdot x_h \cdot \tilde{\phi}(y_h)\| 
\leq 2j \rho_{\tilde{\phi}}(x_h,y_h)
\]
and, recalling Proposition \ref{pro_quasidist}, we obtain
\begin{equation}\label{eq_aux34}
\| \tilde{\phi}(y_h)^{-1}\cdot \phi(x_h)\| \leq 2jC d(\widetilde{p_h},q_h)
\end{equation} 
for some constant $C>0$. Finally
\begin{align*}
\| \tilde{ \phi}(x_h)^{-1}\cdot \phi(x_h)\| &\leq \|\tilde{\phi}(x_h)^{-1}\cdot \tilde{\phi}(y_h)\|+\|\tilde{\phi}(y_h)^{-1}\cdot \phi(x_h) \| \\
&\leq \ve \|x_h\|+2jCd(\tilde{p_h},q_h)\\
&\leq \ve \|x_h\|+2jC \ve \|x_h\|\\
&=\underbrace{(1+2jC)}_{M}\ve\|x_h\|,
\end{align*}
where in the first three lines we used the triangle inequality, \eqref{eq_aux15} and \eqref{eq_aux34}, and, finally, \eqref{eq_aux13}. Combining the latter with \emph{Step 6} (that is $d(\phi(x_h),\tilde{\phi}(x_h)) \geq K\|x_h\|$) we get
\begin{equation}\label{eq_aux35}
\| \tilde{ \phi}(x_h)^{-1}\cdot \phi(x_h)\|  \leq M \ve \frac{d(\phi(x_h),\tilde{\phi}(x_h))}{K}= M \ve \frac{\|\tilde{\phi}(x_h)^{-1} \cdot\phi(x_h)\|}{K}
\end{equation}
where $K$ is the same constant coming from \emph{Step 6}. Simplifying $ \| \tilde{ \phi}(x_h)^{-1}\cdot \phi(x_h)\| $ from both sides of \eqref{eq_aux35} we obtain
\[
1 \leq \frac{M\ve}{K}
\]
but then we get a contradiction from the arbitrariness of $\ve$, concluding the proof.
\end{proof}
\begin{rem}
We emphasize that the proof of Theorem \ref{maindiff} is not dependent on the particular structure of $\he^n$. It can be extended to general Carnot groups $\G = \W \V$, where $\W$ is a normal subgroup and $\V$ is horizontal, provided that a Rademacher-type theorem holds. This is particularly relevant for graphs of codimension 1 either in step 2 Carnot groups or, more generally, in Carnot groups of type $^\star$. However, it is important to note that the validity of a Rademacher-type theorem for intrinsic graphs of arbitrary codimension has only been established in the Heisenberg case. For this reason, we have chosen to set this paper within the context of $\H^n$.

\end{rem}
\section{Alternative proof of Stepanov Theorem for graphs of codimension 1}
In this subsection we study the case of \textit{1-codimensional} graphs in $\h^n=\W\cdot\V$, which means that we restrict to the case dim$(\V)=1$. Under this assumption, $\V=\{\exp (tV):t\in\R\}$ for a fixed $V\in \mathfrak h_1$ (the first layer of the stratification of the Lie algebra of $\h^n$) and we can naturally identify $\V\equiv \R$ with its usual order relation. In particular we can define \textit{infimum} and \textit{supremum} of functions: if $\phi_\alpha:\W\to\V$ are such that $\phi_\alpha(w)=\exp(g_\alpha(w)V)$, for some $g_\alpha:\W\to\R$, we define
\[\inf_{\alpha\in I}\phi_\alpha(w)\ceq \exp\left(\inf_{\alpha\in I}g_\alpha(w)V\right).\]
In the same way we can define the supremum of a family of maps from $\W$ to $\V$.\\
Infimum and supremum of intrinsic Lipschitz maps are themselves intrinsic Lipschitz maps:
\begin{lemma}[{\cite[Proposition 4.24]{fssc2010}}]\label{infLip}
 Let $\W,\V$ be complementary subgroups of $\H^n$ where $\V$ is a horizontal subgroup of dimension $1$.  Then, for all $L>0$ there exists $\widetilde L\geq L$ with the following property: if $\{\phi_\alpha:\W\to\V\}$ is a family of intrinsic $L$-Lipschitz maps, then the function $\phi:=\inf_\alpha \phi_\alpha$ is either well defined and intrinsic $\widetilde L$-Lipschitz, or $\phi\equiv -\infty$. The same property holds for the supremum.
\end{lemma}
Before presenting the alternative proof, inspired by the proof of J. Mal\'y~\cite{maly1999}, we need an auxiliary lemma.
\begin{lemma}\label{lem_carabinieri}
Let $\W,\V$ be complementary subgroups of $\H^n$, where $\V$ is a horizontal subgroup of dimension $1$. Let $A \subset \W$ be open and  $\bar w \in A$. Suppose $\psi,\phi,\eta: A \to \V$ are such that $\psi \leq \phi \leq \eta$ on $A$ (where we identify $\V \equiv \R$), $\psi(\bar w)=\phi(\bar w)=\eta(\bar w)$ and assume also that the functions $\psi$ and $\eta$ are intrinsically differentiable at $\bar w$. Then $\phi$ is intrinsically differentiable at $\bar w$ and 
\[
d\psi_{\bar w} \equiv d\phi_{\bar w} \equiv d\eta_{\bar w}.
\]
\end{lemma}
\begin{proof}
For convenience we split the proof into 4 steps.

\emph{Step 1:} We start by proving that, if a map $\varphi:A\subset \W\to\V$ is intrinsically differentiable at $\bar w=0$, $\varphi(0)=0$ and $\varphi(w)\geq 0$ for every $w\in A$, then $d\varphi_0\equiv 0$. 
Notice that in this case $\varphi_0(w)=\varphi(w)$. By differentiability we know that
\begin{equation}\label{diff}\frac{\|d\varphi_0(w)^{-1}\cdot \varphi(w)\|}{\|w\|}\to 0\quad \text{ if }w\to 0.\end{equation}
Since $\V\equiv\R$ is horizontal, we can rewrite \eqref{diff} as
\[\frac{\varphi(w)-d\varphi_0(w)}{\|w\|}\to 0\quad \text{ if }w\to 0.\]
By definition of limit, for every $\epsilon>0$ there exists $\delta>0$ such that, if $\|w\|<\delta$, then 
\[\frac{\varphi(w)}{\|w\|}<\frac{d\varphi_0(w)}{\|w\|}+\epsilon.\]
Since $\varphi(w)\geq 0$ and by homogeneity of $d\varphi_0$ we get
\[d\varphi_0\left(\delta_{\frac{1}{\|w\|}} w\right)>-\epsilon.\]
Notice that $\delta_{\frac{1}{\|w\|}} w\in \partial B(0,1)\cap \W$, independently on $\delta$. By arbitrariness of $\epsilon$ we infer that $d\varphi_0(u)\geq 0$ for every $u\in\partial B(0,1)\cap \W$, hence also for every $u\in\W$. This is possible only if $d\varphi_0\equiv 0$, because an intrinsic linear map from $\W$ to $\V$ is actually $H$-linear and, if $d\varphi_0(u)>0$ for some $u$, then $d\varphi_0(u^{-1})=-d\varphi_0(u)<0$.

\emph{Step 2}:
We show now that, if $\rho,\sigma:A\subset \W\to\V$ are intrinsically differentiable at 0 with $\rho(0)=\sigma(0)=0$, then $\rho^{-1}\sigma$ is also intrinsically differentiable at 0 and moreover $d(\rho^{-1}\sigma)_0=d\rho_0^{-1}\cdot d\sigma_0$.\\
It is easy to check that $d\rho_0^{-1}\cdot d\sigma_0$ is an intrinsic linear map, since in our setting intrinsic linear maps are exactly $H$-linear maps (see \cite[Proposition 3.23]{as2009}).
So we are left to prove that
\[\frac{\|(d\rho_0^{-1}(w)\cdot d\sigma_0(w))^{-1}\cdot (\rho^{-1}\sigma)_0(w)\|}{\|w\|}\to 0\quad \text{ as }w\to 0.\]
Since $\rho(0)=\sigma(0)=0$, then $(\rho^{-1}\sigma)_0=\rho^{-1}\sigma$. Hence, by commutativity of $\V$ and the triangular inequality,
\[\frac{\|(d\rho_0^{-1}(w)\cdot d\sigma_0(w))^{-1}\cdot (\rho^{-1}\sigma)_0(w)\|}{\|w\|}\leq \frac{\|d\rho_0(w)^{-1}\cdot \rho(w)\|}{\|w\|}+\frac{\|d\sigma_0(w)^{-1}\cdot \sigma(w)\|}{\|w\|}\]
The last two quantities tends to 0 as $w\to 0$ since $\rho,\sigma$ are intrinsically differentiable at 0 and we conclude.

\emph{Step 3}:
Let us now prove that if $\psi$ and $\eta$ are as in the statement of the Lemma then $d\psi_{\bar w} \equiv d\eta_{\bar w}$.
Consider the translated functions
\[\begin{split}\psi_{\bar w}(w)&=\psi(\bar w)^{-1}\psi(\bar w \psi(\bar w)w\psi(\bar w)^{-1}),\\\eta_{\bar w}(w)&=\eta(\bar w)^{-1}\eta(\bar w \eta(\bar w)w\eta(\bar w)^{-1}).\end{split}\]
Since $\psi\leq\eta$ and $\psi(\bar w)=\eta(\bar w)$, we get $\psi_{\bar w}\leq \eta_{\bar w}$ and clearly $\psi_{\bar w}(0)=\eta_{\bar w}(0)=0.$ Notice also that since $\psi$ and $\eta$ are intrinsically differentiable at $\bar w$ so are $\psi_{\bar w}$ and $\eta_{\bar w}$ at 0.
\linebreak Hence, by \emph{Step 2}, the map $\psi_{\bar w}^{-1}\eta_{\bar w}$ is intrinsically differentiable at 0 and $d(\psi_{\bar w}^{-1}\eta_{\bar w})_0=d(\psi_{\bar w})_0^{-1}\cdot d(\eta_{\bar w})_0$. Moreover $\psi_{\bar w}^{-1}\eta_{\bar w}\geq 0$. Thus, by \emph{Step 1}, $d(\psi_{\bar w}^{-1}\eta_{\bar w})_0\equiv 0$, which implies $d(\psi_{\bar w})_0\equiv d(\eta_{\bar w})_0$. Hence, by definition, $d\psi_{\bar w}=d\eta_{\bar w}$.

\emph{Step 4}: Let us finally prove the main conclusion. Let $\psi, \phi, \eta$ be as in the statement. As before we observe that $\psi_{\bar w}\leq \phi_{\bar w} \leq \eta_{\bar w}$. Now let $\theta \ceq d\psi_{\bar w}\equiv d\eta_{\bar w}$. Then for every $w \in \W$ near 0 we have
\[
\frac{\psi_{\bar w}(w)-d\psi_{\bar w}(w)}{d(0,w)}\leq \frac{\phi_{\bar w}(w)-\theta(w)}{d(0,w)}  \leq \frac{\eta_{\bar w}(w)-d\eta_{\bar w}(w)}{d(0,w)}.
\]
Then the left hand side and the right hand side go to $0$ when $w \in B(0,s) \cap \W$ for $s \to 0$. This concludes the proof.
\end{proof}
Now we can present the alternative proof for Stepanov Theorem for 1-codimensional intrinsic graphs, that we restate.
\begin{theorem}\label{codim1}
Let $\W,\V$ be complementary subgroups of $\H^n$ where $\V$ is a horizontal subgroup of dimension $1$. Let $A \se \W$ be an open set and $\phi:A \to \V$. Then $\phi$ is intrinsically differentiable almost everywhere on $S_\phi$.
\end{theorem}

\begin{proof}
    
   Let $\{U_j\}_{j\in\N}$ be an enumeration of all rational balls contained in $A$ such that $\phi$ is bounded on $U_j$ (here we are identifying $\W\equiv \R^{2n})$. 
   Is it clear that $S_\phi \se \bigcup_{i \in \N}U_i$. For each $j \in \N$ we define two intrinsic Lipschitz functions $\eta_j$ and $\psi_j$ on $U_j$ by setting
\begin{align}\label{inf}
    \eta_j(w)\ceq &\inf \lbrace \eta(w): \eta \geq \phi \text{ on }B_j, \Lip(\eta,U_j)\leq j \rbrace,\\
     \psi_j(w)\ceq &\sup \lbrace \psi(w):  \psi \leq \phi \text{ on }B_j, \Lip(\psi,U_j)\leq j \rbrace.
\end{align}
By Lemma \ref{infLip} (combined with extension Theorem in \cite[Proposition~3.4]{VSNS2012} or \cite[Theorem~4.1]{fs2016}), for every $j \in \N$ there exists $\tilde j \geq j$ such that $\eta_j$ and $\psi_j$ are intrinsic $\tilde j$-Lipschitz on $U_j$. Define now
\[
N=\bigcup_{j \in \N} \left\lbrace w \in U_j: \eta_j \text{ or }\psi_j \text{ is not intrinsically differentiable at }w \right\rbrace.
\]
By Theorem \ref{teo_RadVit}, we have that $\leb^{2n}(N)=0$. Let $\bar w \in S_\phi \setminus N$: we will prove that $\phi$ is intrinsically differentiable at $\bar w$, concluding the proof. By definition of $S_\phi$, there exist $\beta>0$ and $r>0$ such that
\[C_\beta(\bar w\cdot\phi(\bar w))\cap \gr_{\phi|_{B_\W(\bar w,r)}}=\{\bar w\cdot\phi(\bar w)\}.\]
Since $\mathbb{V}$ has dimension $1$ we can write $\mathbb{V}=\{\mathrm{exp}(tV)\ |\ t\in\mathbb{R}\}.$ Moreover, the ``positive part'' of the cone $C^+_\beta(\bar w\cdot \phi(\bar w)):= C_\beta(\bar w\cdot \phi(\bar w))\cap \ \mathrm{exp}(\{Z\ |\ \langle Z,V\rangle \geq 0\})$ is the graph of an intrinsic $M$-Lipschitz function $\gamma:\W\to\V$ for some $M>0$ (see Lemma 4.20 in \cite{fssc2010} and also \cite{fs2016}).
Consider now  $i\geq M$ such that
\[
B(\bar w, r/2) \se U_i \se B(\bar w,r).
\]
Clearly $\phi(\bar w)\leq \eta_i(\bar w)$. On the other hand $\gamma$ is a suitable competitor in the family defined in \eqref{inf}: hence $\eta_i(\bar w)\leq \gamma(\bar w)=\phi(\bar w)$. The same argument works for $\psi_i$ and we deduce $\psi_i(\bar w)=\phi(\bar w)=\eta_i(\bar w)$. Hence we conclude using Lemma \ref{lem_carabinieri}.
\end{proof}
\begin{rem}
Similarly to Theorem \ref{stepanov1} and Theorem \ref{maindiff}, the proof of Theorem \ref{codim1} can be extended to a more general context. Specifically, the same approach can be applied to one-codimensional graphs within general Carnot groups that satisfy a Rademacher-type theorem. The equivalence between intrinsic linear and $H$-linear maps (from normal to abelian subgroups), utilized in Lemma \ref{lem_carabinieri}, can be derived as shown in \cite{as2009}, using, for example, \cite[Proposition 3.4]{dd2021}.
\end{rem}

\bibliographystyle{acm}
\bibliography{STEPANOVbib}

\begin{thebibliography}{10}

\bibitem{MR4277829}
{\sc Antonelli, G., and Merlo, A.}
\newblock Intrinsically {L}ipschitz functions with normal target in {C}arnot
  groups.
\newblock {\em Ann. Fenn. Math. 46}, 1 (2021), 571--579.

\bibitem{as2009}
{\sc Arena, G., and Serapioni, R.}
\newblock Intrinsic regular submanifolds in {H}eisenberg groups are
  differentiable graphs.
\newblock {\em Calc. Var. Partial Differential Equations 35}, 4 (2009),
  517--536.

\bibitem{brz2004}
{\sc Balogh, Z.~M., Rogovin, K., and Z\"{u}rcher, T.}
\newblock The {S}tepanov differentiability theorem in metric measure spaces.
\newblock {\em J. Geom. Anal. 14}, 3 (2004), 405--422.

\bibitem{MR1651959}
{\sc Bongiorno, D.}
\newblock Stepanoff's theorem in separable {B}anach spaces.
\newblock {\em Comment. Math. Univ. Carolin. 39}, 2 (1998), 323--335.

\bibitem{CMP24}
{\sc Caravenna, L., Marconi, E., and Pinamonti, A.}
\newblock H\"older regularity of continuous solutions to balance laws and
  applications in the {H}eisenberg group.
\newblock arXiv:2311.14518.

\bibitem{MR1708448}
{\sc Cheeger, J.}
\newblock Differentiability of {L}ipschitz functions on metric measure spaces.
\newblock {\em Geom. Funct. Anal. 9}, 3 (1999), 428--517.

\bibitem{MR3992573}
{\sc Chousionis, V., F\"assler, K., and Orponen, T.}
\newblock Intrinsic {L}ipschitz graphs and vertical {$\beta$}-numbers in the
  {H}eisenberg group.
\newblock {\em Amer. J. Math. 141}, 4 (2019), 1087--1147.

\bibitem{MR4388340}
{\sc Chousionis, V., Li, S., and Young, R.}
\newblock The strong geometric lemma for intrinsic {L}ipschitz graphs in
  {H}eisenberg groups.
\newblock {\em J. Reine Angew. Math. 784\/} (2022), 251--274.

\bibitem{MR3465805}
{\sc Citti, G., Manfredini, M., Pinamonti, A., and Serra~Cassano, F.}
\newblock Poincar\'e-type inequality for {L}ipschitz continuous vector fields.
\newblock {\em J. Math. Pures Appl. (9) 105}, 3 (2016), 265--292.

\bibitem{dedonato2024stepanovtheoremqvaluedfunctions}
{\sc De~Donato, P.}
\newblock The {S}tepanov theorem for {Q}-valued functions, 2024.
\newblock Preprint available online at \url{https://arxiv.org/abs/2402.14554}.

\bibitem{didonatophd}
{\sc Di~Donato, D.}
\newblock Intrinsic differentiability and {I}ntrinsic {R}egular {S}urfaces in
  {C}arnot groups, 2017.
\newblock PhD Thesis available online at
  \url{http://eprints-phd.biblio.unitn.it/2660/1/TesiFinaleDottorato_DiDonatoDaniela.pdf}.

\bibitem{dd2021}
{\sc Di~Donato, D.}
\newblock Intrinsic differentiability and intrinsic regular surfaces in
  {Carnot} groups.
\newblock {\em Potential Anal. 54}, 1 (2021), 1--39.

\bibitem{MR2414208}
{\sc Duda, J.}
\newblock On {G}ateaux differentiability of pointwise {L}ipschitz mappings.
\newblock {\em Canad. Math. Bull. 51}, 2 (2008), 205--216.

\bibitem{federer1969}
{\sc Federer, H.}
\newblock {\em Geometric measure theory}.
\newblock Die Grundlehren der mathematischen Wissenschaften, Band 153.
  Springer-Verlag New York, Inc., New York, 1969.

\bibitem{fms2014}
{\sc Franchi, B., Marchi, M., and Serapioni, R.~P.}
\newblock Differentiability and approximate differentiability for intrinsic
  {Lipschitz} functions in {Carnot} groups and a {Rademacher} theorem.
\newblock {\em Anal. Geom. Metr. Spaces 2\/} (2014), 258--281.

\bibitem{MR1871966}
{\sc Franchi, B., Serapioni, R., and Serra~Cassano, F.}
\newblock Rectifiability and perimeter in the {H}eisenberg group.
\newblock {\em Math. Ann. 321}, 3 (2001), 479--531.

\bibitem{fssc2006}
{\sc Franchi, B., Serapioni, R., and Serra~Cassano, F.}
\newblock Intrinsic {L}ipschitz graphs in {H}eisenberg groups.
\newblock {\em J. Nonlinear Convex Anal. 7}, 3 (2006), 423--441.

\bibitem{fssc2010}
{\sc Franchi, B., Serapioni, R., and Serra~Cassano, F.}
\newblock Differentiability of intrinsic {L}ipschitz functions within
  {H}eisenberg groups.
\newblock {\em J. Geom. Anal. 21}, 4 (2011), 1044--1084.

\bibitem{fs2016}
{\sc Franchi, B., and Serapioni, R.~P.}
\newblock Intrinsic {L}ipschitz graphs within {C}arnot groups.
\newblock {\em J. Geom. Anal. 26}, 3 (2016), 1946--1994.

\bibitem{MR1800917}
{\sc Heinonen, J.}
\newblock {\em Lectures on analysis on metric spaces}.
\newblock Universitext. Springer-Verlag, New York, 2001.

\bibitem{MR2291675}
{\sc Heinonen, J.}
\newblock Nonsmooth calculus.
\newblock {\em Bull. Amer. Math. Soc. (N.S.) 44}, 2 (2007), 163--232.

\bibitem{MR3363168}
{\sc Heinonen, J., Koskela, P., Shanmugalingam, N., and Tyson, J.~T.}
\newblock {\em Sobolev spaces on metric measure spaces}, vol.~27 of {\em New
  Mathematical Monographs}.
\newblock Cambridge University Press, Cambridge, 2015.
\newblock An approach based on upper gradients.

\bibitem{MR4379561}
{\sc Julia, A., Nicolussi~Golo, S., and Vittone, D.}
\newblock Nowhere differentiable intrinsic {L}ipschitz graphs.
\newblock {\em Bull. Lond. Math. Soc. 53}, 6 (2021), 1766--1775.

\bibitem{MR4329286}
{\sc Le~Donne, E., and Moisala, T.}
\newblock Semigenerated {C}arnot algebras and applications to sub-{R}iemannian
  perimeter.
\newblock {\em Math. Z. 299}, 3-4 (2021), 2257--2285.

\bibitem{maly1999}
{\sc Mal\'y, J.}
\newblock A simple proof of the {S}tepanov theorem on differentiability almost
  everywhere.
\newblock {\em Exposition. Math. 17}, 1 (1999), 59--61.

\bibitem{maly2015}
{\sc Mal{\'y}, J., and Zaj{\'{\i}}{\v{c}}ek, L.}
\newblock On {Stepanov} type differentiability theorems.
\newblock {\em Acta Math. Hung. 145}, 1 (2015), 174--190.

\bibitem{MR3815462}
{\sc Naor, A., and Young, R.}
\newblock Vertical perimeter versus horizontal perimeter.
\newblock {\em Ann. of Math. (2) 188}, 1 (2018), 171--279.

\bibitem{MR4460594}
{\sc Naor, A., and Young, R.}
\newblock Foliated corona decompositions.
\newblock {\em Acta Math. 229}, 1 (2022), 55--200.

\bibitem{sc2016}
{\sc Serra~Cassano, F.}
\newblock Some topics of geometric measure theory in {C}arnot groups.
\newblock In {\em Geometry, analysis and dynamics on sub-{R}iemannian
  manifolds. {V}ol. 1}, EMS Ser. Lect. Math. Eur. Math. Soc., Z\"{u}rich, 2016,
  pp.~1--121.

\bibitem{stepanov1923}
{\sc Stepanoff, W.}
\newblock \"{U}ber totale {D}ifferenzierbarkeit.
\newblock {\em Math. Ann. 90}, 3-4 (1923), 318--320.

\bibitem{VSNS2012}
{\sc Vittone, D.}
\newblock Lipschitz surfaces, perimeter and trace theorems for {BV} functions
  in {C}arnot-{C}arath\'eodory spaces.
\newblock {\em Ann. Sc. Norm. Super. Pisa Cl. Sci. (5) 11}, 4 (2012), 939--998.

\bibitem{vittone2021}
{\sc Vittone, D.}
\newblock Lipschitz graphs and currents in {H}eisenberg groups.
\newblock {\em Forum Math. Sigma 10\/} (2022), Paper No. e6, 104.

\bibitem{vodopyanov2000}
{\sc Vodop'yanov, S.~K.}
\newblock {$\mathscr{P}$}-differentiability on {C}arnot groups in different
  topologies and related topics.
\newblock In {\em Proceedings on {A}nalysis and {G}eometry ({R}ussian)
  ({N}ovosibirsk {A}kademgorodok, 1999)\/} (2000), Izdat. Ross. Akad. Nauk Sib.
  Otd. Inst. Mat., Novosibirsk, pp.~603--670.

\bibitem{wz2015}
{\sc Wildrick, K., and Z{\"u}rcher, T.}
\newblock Sharp differentiability results for the lower local {Lipschitz}
  constant and applications to non-embedding.
\newblock {\em J. Geom. Anal. 25}, 4 (2015), 2590--2616.

\end{thebibliography}

\end{document}